\documentclass[12pt]{article}

\usepackage{latexsym,amsmath,amscd,amssymb,framed,graphics,color}
\usepackage{enumerate}
\usepackage{graphicx}
\usepackage{url}
\usepackage[all]{xy}
\usepackage[colorlinks]{hyperref}
\usepackage[square,authoryear]{natbib}

\newcommand{\rem}[1]{}

\makeatletter

\@addtoreset{figure}{section}
\def\thefigure{\thesection.\@arabic\c@figure}
\def\fps@figure{h, t}
\@addtoreset{table}{bsection}
\def\thetable{\thesection.\@arabic\c@table}
\def\fps@table{h, t}
\@addtoreset{equation}{section}

\makeatother

\newcommand \al{\alpha}
\newcommand\be{\beta}
\newcommand\ga{\gamma}

\newcommand\ph{\varphi}

\newcommand\ps{\psi}
\newcommand\om{\omega}
\newcommand\Ga{\Gamma}

\newcommand\Om{\Omega}

\newcommand\resp{resp.\ }
\newcommand\ie{i.e.\ }

\newcommand\oo{{\infty}}

\renewcommand\o{\circ}

\newcommand\x{\times}
\newcommand\on{\operatorname}

\newcommand\Emb{\on{Emb}}

\newcommand\Diff{\on{Diff}}
\newcommand\Gr{\on{Gr}}

\newcommand\vol{\on{vol}}
\newcommand\g{\mathfrak g}
\newcommand\pa{\partial}
\newcommand\h{\mathfrak h}

\newcommand\J{\mathbf{J}}
\newcommand\X{\mathfrak X}

\newcommand\RR{\mathbb R}
\newenvironment{proof}[1][Proof]{\noindent\textbf{#1.} }{\ \rule{0.5em}{0.5em}}

\def\XXint#1#2#3{{\setbox0=\hbox{$#1{#2#3}{\int}$ }
\vcenter{\hbox{$#2#3$ }}\kern-.5\wd0}}

\begin{document}

\newtheorem{theorem}{Theorem}[section]
\newtheorem{definition}[theorem]{Definition}
\newtheorem{lemma}[theorem]{Lemma}
\newtheorem{remark}[theorem]{Remark}
\newtheorem{proposition}[theorem]{Proposition}
\newtheorem{corollary}[theorem]{Corollary}
\newtheorem{example}[theorem]{Example}

\def\below#1#2{\mathrel{\mathop{#1}\limits_{#2}}}

%%%%%%%%%%%%%%%%%%%%%%%%%%%%%%%%%%%%%%%%%%%%%%%%%%%%%%%%%%%%%%%%%%%%%%%%%%%%%%%
%%%%%%%%

%%%%%%%%%%%%%%%%%%%%%%%%%%%%%%%%%%%%%%%%%%%%%%%%%%%%%%%%%%%%%%%%%%%%%%%%%%%%%%%
%%%%%%%%

\title{A dual pair for free boundary fluids}
\author{Fran\c{c}ois Gay-Balmaz and Cornelia Vizman}

\date{ }
\maketitle
\makeatother
\maketitle

\begin{abstract}{ 
We construct a dual pair associated to the Hamiltonian geometric formulation of perfect fluids with free boundaries. 
This dual pair is defined on the cotangent bundle of the space of volume preserving embeddings
of a manifold with boundary into a boundaryless manifold of the same dimension. The dual pair properties are rigorously verified in the infinite dimensional Fr\'echet manifold setting. It provides an example of a dual pair associated to actions that are not completely mutually orthogonal.}
\end{abstract}

\let\thefootnote\relax
\let\thefootnote\relax

%%%%%%%%%%%%%%%%%%%%%%%%%%%%%%%%%%%%%%%%
%\tableofcontents

%%%%%%%%%%%%%%%%%%%%%%%%%%

\section{Introduction}

The concept of dual pair, formalized by \cite{We83}, is an important notion in Poisson geometry and has several applications in the context of
momentum maps and reduction theory, see e.g. \cite{OrRa04} and references therein. We recall that given a finite dimensional symplectic manifold $(M,\omega)$ and two finite dimensional Poisson manifolds $P_1, P_2$, a pair of Poisson mappings
\begin{equation*}
P_1\stackrel{\J_1}{\longleftarrow}(M,\om)\stackrel{\J_2}{\longrightarrow} P_2
\end{equation*}
is called a \textit{dual pair\/} if $\ker T\J_1$ and $\ker T\J_2$ are symplectic orthogonal complements of one another, where $\ker T\J_i$ denotes the kernel of the tangent map $T\mathbf{J}_i$ of $\mathbf{J}_i$.
Dual pair structures arise naturally in classical mechanics. In many cases, the Poisson maps $\mathbf{J}_i$ are momentum mappings associated to Lie algebra actions on $M$.
For example, in \cite{Ma1987} (see also \cite{CuRo1982}, \cite{GoSt1987} 
and \cite{Iw1985}) it was shown that the concept of dual pair of momentum maps can be useful for the study of bifurcations in Hamiltonian systems with symmetry.

In the context of infinite dimensional manifolds, there are several difficulties related with the definition of dual pairs and rigorous justification of their properties in concrete examples. Such difficulties have been overcome in \cite{GBVi2012} for the case of dual pairs of momentum maps associated to the Euler equations of a perfect fluid and the dual pairs of momentum maps associated to the $n$-dimensional Camassa-Holm equation (also named EPDiff equations, standing for Euler-Poincar\'e equations on the diffeomorphism group). These dual pairs have been defined in \cite{MaWe83} and \cite{HoMa2004}, respectively, and clarify, among other things, the occurrence of singular solutions as Clebsch variables (in the sense of \cite{MaWe83}) together with their gauge group symmetries. 

\medskip

We now briefly recall the construction of these two dual pairs structures at a formal level, that is, without describing the additional restrictions and reformulations obtained in \cite{GBVi2012} to prove the dual pair property. 

For ideal fluids, the dual pair is constructed as follows.  
Consider a symplectic manifold $(M,\omega)$, a volume manifold $(S,\mu)$, and let $C^\infty(S,M)$ be the space of smooth maps from $S$ to $M$. The left action of the group $\operatorname{Diff}(M,\omega)$ of symplectic diffeomorphisms and the right action of the group $\operatorname{Diff}(S,\mu)$ of volume preserving diffeomorphisms are two commuting symplectic actions on $C^\infty(S,M)$. Their momentum maps $\mathbf{J}_L$ and $\mathbf{J}_R$
form the dual pair for the Euler equation:
\[
\X(M, \omega )^*\stackrel{\J_L}{\longleftarrow} C^\infty(S,M)\stackrel{\J_R}{\longrightarrow}\X(S, \mu )^*.
\]
While the right leg represents Clebsch variables for the Euler equations seen as a Hamiltonian system on $\X(S, \mu )^*$, the left leg is a
constant of motion for the induced Hamiltonian system on $C^\infty(S, M)$, which reduces to point vortex solutions of the two dimensional Euler equations when $\operatorname{dim}(S)=0$ and $M$ is two dimensional (see \cite{MaWe83}). As emphasized in \cite{HoTr2009}, this dual pair naturally arises also in kinetic theory, in relation with the Klimontovich solutions of the Poisson-Vlasov system in plasma physics.

To obtain the dual pair of the $n$-dimensional Camassa-Holm equation one considers the cotangent bundle $T^* \operatorname{Emb}(S,M)$ of the manifold $ \operatorname{Emb}(S,M)$ of all embeddings of $S$ into $M$. The diffeomorphism groups $ \operatorname{Diff}(S)$ and $ \operatorname{Diff}(M)$ naturally act on $T^* \operatorname{Emb}(S,M)$ by the cotangent lift of right and left composition, respectively. The associated momentum maps form the dual pair for the $n$-Camassa-Holm equation:
\[
\X(M)^*\stackrel{\J_L}{\longleftarrow}T^* \operatorname{Emb}(S,M)\stackrel{\J_R}{\longrightarrow}\X(S)^*,
\]
\cite{HoMa2004}.
While the left leg provides singular solutions of the EPDiff equation, the right leg
is a constant of motion associated to the collective motion on $T^*\operatorname{Emb}(S, M)$.

\medskip

In this paper we consider and study a dual pair associated to an other important example of infinite dimensional Hamiltonian system arising in fluid mechanics, namely, the case of a perfect fluid with free boundaries.

As we will recall in details later, the phase space of a free boundary fluid is given by the cotangent bundle $T^* \operatorname{Emb}_{\vol}(S,M)$ of the manifold $ \operatorname{Emb}_{\vol}(S,M)$ of all volume preserving embeddings of a compact volume manifold $(S,\mu_S)$ with boundary $\pa S$ into a volume manifold $(M,\mu)$ without boundary, both having the same dimension. The Hamiltonian structure of free boundary perfect fluids has been shown in \cite{LeMaMoRa1986} to arise by Poisson reduction of the canonical symplectic form on $T^* \operatorname{Emb}_{\vol}(S,M)$  by the group $ \operatorname{Diff}_{\vol}(S)$ of all volume preserving diffeomorphisms $ \operatorname{Diff}_{\vol}(S)$ of $S$. The reduced space is therefore given by the quotient $T^* \operatorname{Emb}_{\vol}(S,M)/ \operatorname{Diff}_{\vol}(S)$ endowed with the reduced Poisson bracket. On the other hand, Noether theorem associated to the $\operatorname{Diff}_{\vol}(S)$-action gives rise to conservation of the momentum map $ \mathbf{J}_R: T^* \operatorname{Emb}_{\vol}(S,M) \rightarrow \mathfrak{X} _{\vol,\|}(S)^\ast $. This momentum map, together with the quotient map $ \pi_R $, form the dual pair associated to free boundary fluid that will be the main object of study in this paper:
\begin{equation}\label{DP_free_boundary} 
T^* \operatorname{Emb}_{\vol}(S,M)/ \operatorname{Diff}_{\vol}(S)\stackrel{\pi_R }{\longleftarrow}T^* \operatorname{Emb}_{\vol}(S,M)\stackrel{\J_R}{\longrightarrow}\X_{\vol,\|}(S)^*.
\end{equation} 

We will obtain the dual pair property of \eqref{DP_free_boundary} by first considering the pair of momentum maps
\[
\X_{\vol}(M)^*\stackrel{\J_L}{\longleftarrow}T^*\Emb_{\vol}(S,M)^\x\stackrel{\J_R}{\longrightarrow}\X_{\vol,\|}(S)^*
\]
for which the commuting actions of $\Diff_{\vol}(S)$ and $\Diff_{\vol}(M)$ are not mutually completely orthogonal (in the sense of \cite{LM}),
namely only the $\Diff_{\vol}(S)$-orbits are the symplectic orthogonals of the $\Diff_{\vol}(M)$-orbits, not vice-versa. {Of course, such a situation can only happen in the infinite dimensional case}.

\paragraph{Acknowledgements.}
This work was partially supported by a grant of the Romanian National Authority for Scientific  Research, CNCS UEFISCDI, project number PN-II-ID-PCE-2011-3-0921.

\section{Free boundary fluids}

{In this section we quickly recall the geometric formulation of perfect fluids with fixed or free boundary, and the associated process of reduction by symmetry.}

\paragraph{Geometry of perfect fluids.} The equations of motion of an ideal incompressible fluid on an
oriented compact Riemannian manifold $(M,g)$ with smooth boundary $\partial M$ are given by the Euler equations
\begin{equation}\label{Euler}
\partial _t v+\nabla_vv =-\operatorname{grad} p,
\end{equation}
where the Eulerian velocity $v$ is a divergence free vector
field parallel to the boundary, $p$ is the \hbox{pressure}, and $\nabla$ is the
Levi--Civita covariant derivative associated to $g$.
Arnold \cite{Arnold1966} has shown that equations \eqref{Euler}
are formally the spatial representation of geodesics
on the volume-preserving diffeomorphism group
$\operatorname{Diff}_{\rm vol}(M)$ of $M$ with respect to the $L^2$
Riemannian metric. From this, one obtains that the Euler equations in Lagrangian representation are given by canonical Hamilton's equations on the cotangent bundle $T^* \operatorname{Diff}_{\rm vol}(M)$. The noncanonical Hamiltonian structure of the  Euler equations in spatial representation \eqref{Euler} are obtained by Poisson reduction of the canonical Poisson brackets on $T^* \operatorname{Diff}_{\rm vol}(M)$, as explained in \cite{MaWe83}.
This is a particular instance of the process of Lie--Poisson reduction valid on any Lie group $G$, \cite{MaRa99}. More precisely, the Lie--Poisson brackets
\[
\{f,g\}(\mu) = \left\langle \mu ,\left[\frac{\delta f}{\delta \mu }, \frac{\delta g}{\delta \mu} \right] \right\rangle , \quad \mu \in \mathfrak{g}  ^\ast ,
\]
on the dual space $ \mathfrak{g}  ^\ast $ of the Lie algebra $\mathfrak{g}  $ of $G$, are obtained by Poisson reduction of the canonical Poisson bracket on the cotangent bundle $T^*G$, via the quotient map $T^*G \rightarrow \mathfrak{g}  ^\ast $, $  \alpha _g \mapsto  \alpha _g g ^{-1} $. For the Euler equations, one chooses $ G= \operatorname{Diff}_{\rm vol}(M)$ so that $ \mathfrak{g}  = \mathfrak{X}_{\rm vol,\|}(M)$ is the Lie algebra of divergence free vector fields on $M$ parallel to the boundary. 

\paragraph{Geometry of perfect free boundary fluids.} The above reduced Hamiltonian formulation has been extended to the case of perfect fluids with free boundaries in \cite{LeMaMoRa1986}. The configuration space is given by the manifold $ \operatorname{Emb}_{\vol}(S,M)$ of all volume preserving embeddings of a compact volume manifold $(S, \mu_S  )$ with smooth boundary into a boundaryless Riemannian manifold $(M,g)$ of same dimension. Recall that an embedding $ \varphi :S \rightarrow M$ is volume preserving if $ \varphi ^\ast \mu_ M = \mu _S$, where $ \mu_M$ is the Riemannian volume form on $M$.

The tangent space $T_ \varphi \operatorname{Emb}_{\vol}(S,M)$ consists of vector fields $v_ \varphi : S \rightarrow TM$ along $ \varphi $ such that $v_ \varphi \circ \varphi ^{-1} \in \mathfrak{X}_{\vol}( \varphi (S))$, the space of divergence free vector fields (relative to $ \mu _M$) on $ \varphi (S)$, not necessarily parallel to the boundary.
The quotient space $\operatorname{Emb}_{\vol}(S, M)/ \operatorname{Diff}_{\vol}(S)$ is the nonlinear Grassmannian $ \operatorname{Gr}_0^S(M)$ of all type $S$ submanifolds of $M$ of same volume with $S$, that is,
\begin{equation}\label{defi}
\operatorname{Gr}_0^S(M)=\{D\subset M: \text{$D$ submanif. diffeom. to $S$, $\operatorname{Vol}_{ \mu _M }(D)= \operatorname{Vol}_{ \mu _S }(S)$}\}.
\end{equation}
The quotient space $T \operatorname{Emb}_{\vol}(S, M)/ \operatorname{Diff}_{\vol}(S)$ is identified with the vector bundle over $ \operatorname{Gr}^S_0(M)$  whose fiber at $D$ is given by the space $ \mathfrak{X}_{\rm vol}(D)$ defined above. Using this identification, the projection onto the quotient reads
\begin{equation}\label{quotient_map} 
v_ \varphi \in T \operatorname{Emb}_{\vol}(S, M) \mapsto v  =v_ \varphi \circ \varphi ^{-1} \in T \operatorname{Emb}_{\vol}(S, M)/ \operatorname{Diff}_{\rm vol}(S).
\end{equation}
{We refer to \cite{GBVi2014} for a detailed study of the Fr\'echet manifold structures of nonlinear Grassmannians of type $S$, when $S$ has a nonempty boundary}.

\medskip

The dynamics of the fluid in Lagrangian representation is described by a curve $ \varphi (t) \in \operatorname{Emb}_{\vol}(S,M)$ which indicates the current position $x= \varphi (t)(s)$ at time $t$ of the material particle labelled by $s \in S$ . The free  boundary motion $ \Sigma (t)$ is given by the image of $ \partial S$ by $ \varphi (t)$, that is $ \Sigma (t)= \varphi (t)( \partial S)= \partial ( \varphi (t)(S))$. 

The material Lagrangian $ L: T \operatorname{Emb}_{\vol}(S,M) \rightarrow \mathbb{R}$  of the free boundary fluid reads
\[
L(\varphi , v _\varphi  )= \frac{1}{2} \int_ S g( \varphi (s))(v _\varphi  (s), v _\varphi  (s)) \mu _S(s)  - \tau \int_{ \partial S}\gamma _{ \varphi ^\ast g}(s),
\]
where the constant $ \tau >0$ is the surface tension and $\gamma  _{ \varphi ^\ast g}$ is the volume form on $ \partial S$ induced by the restriction to $ \partial S$ of the metric $ \varphi ^\ast g$ on $S$.
This Lagrangian is $ \operatorname{Diff}_{\vol}(S)$-invariant and  therefore induces a Lagrangian on the quotient space $ T \operatorname{Emb}_{\vol}(S, M)/ \operatorname{Diff}_{\vol}(S)$, the projection map being given by \eqref{quotient_map}. 
The reduced Lagrangian thus reads
\[
\ell(v)= \frac{1}{2} \int_ {D}g(x)(v(x), v (x)) \mu _M (x)- \tau \int_ \Sigma \gamma _ \Sigma (x),
\]
where $v\in\X_{\vol}(D)$, $ \Sigma = \partial D$ and $ \gamma _ \Sigma $ is the volume form on $ \Sigma $ induced by $g$.
The equations of motions are given by
\[
\left\{
\begin{array}{l}
\displaystyle\vspace{0.2cm}\partial _t v + \nabla _ vv  = - \operatorname{grad}p\;\; \;\text{on $D$}\\
\displaystyle\vspace{0.2cm}\partial _t \Sigma = g( n, v ) \\
\displaystyle\operatorname{div}_ {\mu_M} v=0, \;\; p|_{ \Sigma }= \tau \kappa_\Sigma  ,
\end{array}
\right.
\]
where $ \kappa _\Sigma $ is the mean curvature of the hypersurface $ \Sigma $ relative to the metric $g$, and $n$ is the outward-pointing unit normal vector field to $ \Sigma $ relative to $g$.
{The second equation above is written in the tangent space to $ \operatorname{Gr}^S_0 (M)$ at $D$, identified with the space of smooth functions $f $ on $ \Sigma =\partial D$ with $\int_ \Sigma f \gamma _\Sigma =0$.}

We refer to \cite{LeMaMoRa1986} for further informations concerning the geometry of free boundary perfect fluids, the description of the Poisson brackets, and the reduced Hamiltonian formulation. The Lagrangian side together with the associated variational principle for free boundary continuum mechanics has been further studied in \cite{GBMR2012}.

%%%%%%%%%%%%%%%%%%%%%%

\section{Dual pairs in infinite dimensions}\label{deux}

{In this section we review the definition of a dual pair and highlight some difficulties arising in the infinite dimensional case. We then provide general conditions which guarantee the dual pair property in the infinite dimensional case. These results will be applied to the case of free boundary fluids in Section \ref{5}.}

\medskip

Let $(M,\omega)$ be a symplectic manifold and $P_1, P_2$ be two finite dimensional Poisson manifolds. A pair of Poisson mappings
\begin{equation}\label{pois}
P_1\stackrel{\J_1}{\longleftarrow}(M,\om)\stackrel{\J_2}{\longrightarrow} P_2
\end{equation}
is called a \textit{dual pair} (\cite{We83})
if the kernels $\ker T\J_1$ and $\ker T\J_2$ are symplectic orthogonal complements of one another. That is, for each $m\in M$,
\begin{equation}\label{strong_dual_pair}
(\ker T_m\J_1)^\om=\ker T_m\J_2,\quad (\ker T_m\J_2)^\om=\ker T_m\J_1.
\end{equation}
In finite dimensions these two conditions are equivalent, so one is enough to get a dual pair,
as in the example below.

\begin{example}[Dual pair associated to a free and proper Hamiltonian action]
{\rm  Let $G$ be a Lie group acting symplectically (on the left) on a symplectic manifold $(M,\omega)$ and admitting a momentum map $\mathbf{J}:M\rightarrow\mathfrak{g}^*$.
This means that $\mathbf{d}\langle\mathbf{J},\xi\rangle=\mathbf{i}_{\xi_M}\omega$ for all 
$\xi\in \mathfrak{g}$, where $\xi_M$ denotes the infinitesimal generator.
The kernel  of the tangent map $T\mathbf{J}$ is characterized by the equality
\begin{equation}\label{ker_range}
\operatorname{ker}T_m\mathbf{J}=\mathfrak{g}_M(m)^\omega,
\end{equation}
where $\mathfrak{g}_M(m):=\{\xi_M(m)\mid \xi\in\mathfrak{g}\}$. 
%This shows that, when the action is free, $\mathbf{J}$ is a submersion onto an open subset of $\mathfrak{g}^*$. 
When the momentum map is equivariant, it is a Poisson map with respect to the symplectic form on $M$ 
and the $(+)$ Lie--Poisson structure
on $\mathfrak{g}^*$. If we suppose in addition
that $G$ acts freely and properly on $M$, then the quotient space
$M/G$ is a smooth manifold such that the projection
$\pi:M\rightarrow M/G$ is a smooth surjective submersion. This map
is Poisson with respect to the symplectic form on $M$ and the
induced quotient Poisson structure on $M/G$. Thus, using the
equality
$\operatorname{ker}T\mathbf{J}=\left(\mathfrak{g}_M\right)^\omega=(\operatorname{ker}T\pi)^\omega$,
we obtain the dual pair  (\cite{We83})
\begin{equation}\label{reduction}
M/G\stackrel{\pi}{\longleftarrow}(M,\om)\stackrel{\mathbf{J}}{\longrightarrow} \mathfrak{g}^*.
\end{equation}

Let us see what happens in an infinite dimensional setting, 
with  $G$ a connected Fr\'echet Lie group acting freely on the
symplectic Fr\'echet manifold $(M,\om)$ and admitting an equivariant
momentum map $\J:M\to\g^*$. By $\mathfrak{g}^*$ we denote a topological vector space
in non-degenerate duality with the Fr\'echet Lie algebra
$\mathfrak{g}$. When $ \mathfrak{g}  $ is a Lie algebra 
of sections of a vector bundle, standard examples for $\mathfrak{g}^*$ are the full
distributional dual or the regular dual.
Note that the existence of
a momentum map $\mathbf{J}:M\to\g^*$ may depend on the chosen dual
$\mathfrak{g}^*$. {Note also that equality \eqref{ker_range} still holds in the infinite dimensional case}.

Assuming that the quotient space $M/G$ can be endowed with a smooth manifold structure such that the projection $\pi$ is a smooth map, we have $\mathfrak{g}_M=\operatorname{ker}(T\pi)$,
thus we get the equality
$(\ker T\pi)^\om= \ker T\J$ and the inclusion $\ker T\pi\subset(\ker T\J)^\om$.
Therefore, contrary to the finite dimensional case, in infinite dimensions one cannot conclude that
\eqref{reduction} is a dual pair.$\qquad\blacklozenge$ 
}
\end{example}

A weaker notion of dual pair was introduced in \cite{GBVi2012}, where dual pairs associated to
the ideal fluid and EPDiff equations were studied.
A pair of Poisson mappings (\ref{pois}) is called a {\it weak dual pair},
if $\ker T\J_1$ and $\ker T\J_2$
satisfy the inclusions
\begin{equation}\label{dual_pair}
(\ker T\J_1)^\om\subset\ker T\J_2,\quad (\ker T\J_2)^\om\subset\ker T\J_1.
\end{equation}
In finite dimensions these two inclusions are equivalent. 
An example of weak dual pair is provided by the cotangent momentum maps for two commuting Hamiltonian actions, 
as we will see below.

\begin{example}[Invariant momentum maps]
{\rm
We consider two symplectic actions of the (possibly infinite dimensional Fr\'echet) Lie algebras $\mathfrak{g}$ and $\mathfrak{h}$
on the (possibly infinite dimensional Fr\'echet) symplectic manifold $(M,\omega)$, $\g$ acting on the right and $\h$ on the left.
We assume these actions admit infinitesimally equivariant momentum maps $\mathbf{J}_R$ and $\mathbf{J}_L$,
hence $\J_L:M\to(\h^*,\{\ ,\ \}_+)$ \resp $\J_R:M\to(\g^*,\{\ ,\ \}_-)$ are formally Poisson maps
for the plus \resp minus Lie--Poisson brackets. 

Suppose that $\mathbf{J}_L$ is infinitesimally invariant under the action of $\mathfrak{g}$, i.e., $ \mathfrak{g}  _M \subset \operatorname{ker}T \mathbf{J} _L$ or, equivalently, $ \mathfrak{g}  _M \subset (\mathfrak{h}  _M) ^\omega $ by \eqref{ker_range}. This is equivalent to the property $\omega(\xi_M,\eta_M)=0$, for all $\xi\in\mathfrak{g}$ and $\eta\in\mathfrak{h}$.
This turns out to be also equivalent to the infinitesimal invariance of $ \mathbf{J} _R $ under the action of $ \mathfrak{h}  $.

Passing to the symplectic orthogonal spaces, we get that
$\operatorname{ker}(T\mathbf{J}_L)^\om\subset\left(\mathfrak{g}_M\right)^\omega
=\operatorname{ker}(T\mathbf{J}_R)$ and $\operatorname{ker}(T\mathbf{J}_R)^\om\subset\left(\mathfrak{h}_M\right)^\omega=\operatorname{ker}(T\mathbf{J}_L)$,
hence the pair of momentum maps 
\begin{equation}\label{dp1}
\h^*\stackrel{\J_L}{\longleftarrow}(M,\om)\stackrel{\J_R}{\longrightarrow}\g^*
\end{equation}
is a weak dual pair. The situation can be summarized by the following diagram
\begin{displaymath}
\begin{xy}
\xymatrix{
\mathbf{J}  _R\text{ is $\mathfrak{h}$-inv.} \ar@{<=>}[r]  &\mathfrak{h}  _M \subset ( \mathfrak{g}  _M ) ^\omega  \ar@{<=>}[rr]\ar@{=>}[dd] \ar@{=>}[dr]&\hspace{-0.5cm} &\hspace{-0cm}\mathfrak{g}  _M \subset (\mathfrak{h}  _M) ^\omega  \ar@{=>}[ld]\ar@{=>}[dd]&\hspace{-0.6cm}\phantom{hoo}\mathbf{J}  _L\text{ is $\mathfrak{g}$-inv.} \ar@{<=>}[l]\\
& & \eqref{dual_pair}\ar@{=>}[ld]\ar@{=>}[rd] & &\\
&\hspace{-1cm}(\operatorname{ker}T \mathbf{J} _R) ^\omega \subset \operatorname{ker}T \mathbf{J} _L   & & \hspace{-0.5cm}(\operatorname{ker}T \mathbf{J} _L) ^\omega \subset \operatorname{ker}T \mathbf{J} _R.&
}\end{xy}
\end{displaymath}
In the finite dimensional case, all the arrows become equivalences.
$\qquad\blacklozenge$ }
\end{example}

\begin{lemma}\label{one_is_enough} Let \eqref{pois} be a weak dual pair. If one of the inclusions in \eqref{dual_pair} is an equality then both inclusions are equalities, i.e., \eqref{strong_dual_pair} holds.
\end{lemma} 
\begin{proof} Suppose that $(\ker T\J_1)^\om=\ker T\J_2$ and $(\ker T\J_2)^\om\subset\ker T\J_1$. We thus have $(\ker T\J_1)^{\om\om}=(\ker T\J_2) ^\omega \subset \ker T\J_1$. But since we aways have the inclusion $\ker T\J_1 \subset (\ker T\J_1)^{\om\om}$ (in the finite dimensional case, this is an equality), it follows that $(\ker T\J_2)^\om=\ker T\J_1$.
\end{proof} 

\medskip

From this Lemma, we obtain the following improvement of a result of \cite{GBVi2012}. It will be needed later to establish the dual pair property arising in the context of free boundary fluids.

\color{black}

\medskip

\begin{proposition}\label{modulog}Consider two symplectic actions of (possibly infinite dimensional) Lie algebras $\mathfrak{g}$ and $\mathfrak{h}$ on a (possibly infinite dimensional) symplectic manifold $(M,\omega)$. Assume that these actions admit infinitesimally equivariant momentum maps $\mathbf{J}_R$ and $\mathbf{J}_L$. If
\begin{equation}\label{DP_Prop} 
\h^*\stackrel{\J_L}{\longleftarrow}(M,\om)\stackrel{\J_R}{\longrightarrow}\g^*
\end{equation} 
is a weak dual pair with $\g_M=(\h_M)^\om$, then \eqref{DP_Prop} is a dual pair.

Moreover, assuming that the right momentum map is associated to 
a symplectic Lie group action of $G$ on $M$ and
 that the quotient space $M/G$ can be endowed with a smooth manifold structure such that  the projection $\pi:M\to M/G$ is a smooth map, then
\[
M/G\stackrel{\pi}{\longleftarrow}(M,\om)\stackrel{\mathbf{\J}_R}{\longrightarrow} \mathfrak{g}^*
\]
%\eqref{reduction} 
is a dual pair too.
\end{proposition}

\begin{proof}
From $\g_M=(\h_M)^\om$ we get $\ker T\J_R=(\ker T\J_L)^\om$ by taking the symplectic orthogonals. From Lemma \ref{one_is_enough} we obtain 
also $\ker T\J_L=(\ker T\J_R)^\om$, so the first pair of Poisson maps is a dual pair.

For the second part, $\g_M=(\h_M)^\om$ implies that
$(\ker T\J_R)^\om= \ker T\J_L= (\mathfrak{h}  _M ) ^\omega 
= \mathfrak{g}  _M = \ker T\pi$. This ensures, together with 
$(\ker T\pi)^\om=(\g_M)^\om=\ker T\J_R$, that the second pair of Poisson maps is a dual pair.
\end{proof}

\medskip

\begin{remark}\label{mut_compl_orth} {\rm It should be noted that the equality $\g_M=(\h_M)^\om$ can be written $\mathfrak{g}_M=\operatorname{ker}T\mathbf{J}_L$ and means that $\mathfrak{g}$ acts transitively on the level set of $\mathbf{J}_L$. Similarly, the equality $\h_M=(\g_M)^\om$ means that 
$\mathfrak{h}$ acts transitively on the level set of $\mathbf{J}_R$.

{When  $\g_M=(\h_M)^\om$ and $\h_M=(\g_M)^\om$, the $ \mathfrak{g}$ and $\mathfrak{h}$ actions are said to be {\it mutually completely orthogonal} (\cite{LM}). Of course, in the finite dimensional case the two equalities are equivalent.}
}
\end{remark}

\begin{remark}\label{oneisenough} {\rm In the infinite dimensional case, under the weak dual pair assumption \eqref{dual_pair}, we can write the following diagram that shows the various implications
\begin{displaymath}
\begin{xy}
\xymatrix{
\mathfrak{g}  _M =( \mathfrak{h}  _M )^\omega\ar@<1ex>[d]|{\mathsf{X} } \ar@{=>}[r]
&(\ker T\J_L)^\om=\ker T\J_R
\ar@{<=>}[d]^{\text{Lemma \ref{one_is_enough}}}&\\
\mathfrak{h}  _M = (\mathfrak{g}  _M )^\omega \ar@<1ex>[u]|{\mathsf{X} }\ar@{=>}[r]&(\ker T\J_R)^\om=\ker T\J_L.&
}
\end{xy}
\end{displaymath}
In the finite dimensional case, all the arrows in the diagram become equivalences.
In infinite dimensions, although the equality $\mathfrak{g}  _M =( \mathfrak{h}  _M )^\omega$ implies the dual pair property, it does not imply the equality $\mathfrak{h}_M = (\mathfrak{g}  _M )^\omega$ in general. 
In Section \ref{5} we will encounter a dual pair of momentum maps which satisfies $\mathfrak{g}  _M =( \mathfrak{h}  _M )^\omega$ and does not satisfy $\mathfrak{h}_M = (\mathfrak{g}  _M )^\omega$. 
{This means that this dual pair is not associated to mutually completely orthogonal actions. Of course, such situation can only happen in the infinite dimensional case.}
}
\end{remark}

We now suppose that the Lie algebra actions are associated to symplectic Lie group actions of the Lie groups $G$ and $H$. We also suppose that $\mathbf{J}_L$ is $G$-invariant. Then we have the infinitesimal invariances $\mathfrak{g}_M\subset \operatorname{ker}T\mathbf{J}_L=(\h_M)^\om$ and 
$\mathfrak{h}_M\subset \operatorname{ker}T\mathbf{J}_R=(\g_M)^\om$. 
If $H$ is connected, then $\mathbf{J}_R$ is also $H$-invariant.
The following two corollaries of Proposition \ref{modulog} are stronger versions of  Corollary 2.6 and Corollary 2.8 
of \cite{GBVi2012}. {These corollaries are relevant if $G$, $H$, and $M$ are infinite dimensional}.

\begin{corollary}\label{cipolla}
Let $\J_R$ and $\J_L$ be equivariant momentum maps arising from the symplectic actions of two Lie groups $G$ and $H$ on a symplectic manifold $(M,\om)$. Assume that $\J_L$ is $G$-invariant $($or $\J_R$ is $H$-invariant$)$,
then the pair of momentum maps \eqref{dp1}
is a  weak dual pair. 
Moreover, if the $G$ action is transitive on level sets of $\J_L$ $($or the $H$ action is transitive on level sets of $\J_R$$)$, then \eqref{dp1} is a dual pair.
If the groups $G$ and $H$ are connected, then their actions commute.
\end{corollary}

\begin{corollary}\label{finocchio}
Consider the commuting actions of two Lie groups $G$ and $H$ on a manifold $Q$, and their lift to the cotangent bundle $T^*Q$. Then the associated pair of cotangent momentum maps
\begin{equation}\label{dp3}
\h^*\stackrel{\J_L}{\longleftarrow}T^*Q\stackrel{\J_R}{\longrightarrow}\g^*
\end{equation}
is a weak dual pair. If moreover the $G$ action is transitive on level sets of $\J_L$
$($or the $H$ action is transitive on level sets of $\J_R$$)$, then \eqref{dp3} is a dual pair. 
\end{corollary}

\begin{proof}
If $M=T^*Q$ is a cotangent bundle on which $G$ acts by the cotangent lift of its action on $Q$, then
\begin{equation}\label{cot_momap}
\mathbf{J}:T^*Q\rightarrow\mathfrak{g}^*,\quad \langle\mathbf{J}(\alpha_q),\xi\rangle=\langle\alpha_q,\xi_Q(q)\rangle
\end{equation}
is an equivariant momentum map \cite{MaRa99}. If $\xi_Q$ is an infinitesimal generator of the $G$ action, then $\xi_Q$ is $H$-equivariant, since the actions commute. Thus, the cotangent momentum map $\mathbf{J}_R$, associated to the cotangent lifted action of $G$, is $H$-invariant. Both the weak dual pair and dual pair properties
follow now by Corollary \ref{cipolla}.
\end{proof}

%%%%%%%%%%%%%%%%%%%%%%%

\section{Volume preserving embeddings}

{In this section we present the infinite dimensional manifolds involved in the geometric formulation of free boundary fluids and comment on their Fr\'echet differential structures.}

\paragraph{Manifolds of embeddings.}
Let $S$ be a compact manifold with smooth boundary and let $M$ be a manifold without boundary.
We denote by $C^\infty(S,M)$ the Fr\'echet manifold of all smooth functions 
$\ph:S\to M$. Its tangent space at $\ph$ consists of all smooth vector fields along $\varphi $, i.e., $T_\varphi C^\infty(S,M)=\{ u_ \varphi \in C^\infty(S,TM) : u_ \varphi(s) \in T_{\varphi(s)}M,\;\forall\,s \in S\}$.

The set $\Emb(S,M)$ of embeddings  is an open subset of $C^\infty(S,M)$. The left action of the Lie algebra $\X(M)$ of vector fields on $\Emb(S,M)$  is transitive, 
i.e., each $ v_ \varphi \in T_\varphi \Emb(S,M)$ can be written as $v_ \varphi =v\o \varphi$ for some $v\in\X(M)$.
Note that $v$ is not necessarily parallel to the boundary of $ \varphi (S)$.

\medskip

{For the rest of this section we assume that $ \operatorname{dim}S= \operatorname{dim}M=n$.
Suppose that $S$ and $M$ are endowed with volume forms $ \mu _S $ and $ \mu _M $. The set of volume preserving embeddings 
\[
\Emb_{\vol}(S,M):=\{\varphi \in\Emb(S,M):\varphi ^*\mu_M=\mu_S\}
\]
is a manifold  \cite{GBVi2014} with tangent space
\begin{equation}\label{Tangent_space_vol} 
T_\varphi \Emb_{\vol}(S,M)=\{v_\varphi \in T_\varphi \Emb(S,M):\operatorname{div}_{\mu _M }( v_ \varphi \circ \varphi ^{-1} )=0\},
\end{equation} 
where $v_\ph\o\ph^{-1}$ is a vector field in $\X(\ph(S))$ and $\operatorname{div}_{\mu _M }$ denotes the divergence relative to the volume form $ \mu _M $.
}

\paragraph{Actions of diffeomorphisms groups.}
Consider the left and right actions of the diffeomorphism groups $\Diff(M)$ and $\Diff(S)$ on $\Emb(S,M)$ given by
\[
(\eta ,\varphi )\mapsto \eta \o \varphi \quad\text{and}\quad (\varphi ,\ps)\mapsto \varphi \o\ps, \quad \eta \in \operatorname{Diff}(M), \quad \psi \in \operatorname{Diff}(S).  
\]
These actions restrict to left and right actions of the subgroups of volume preserving diffeomorphisms
$\Diff_{\vol}(M)$ and $\Diff_{\vol}(S)$ on the manifold of volume preserving embeddings $\Emb_{\vol}(S,M)$. Let us denote by
\begin{align*} 
\X_{\vol}(M)&=\{ v \in \mathfrak{X}  (M) : \operatorname{div}_{\mu _M } v=0 \}\\
\X_{\vol,\|}(S)&=\{ w \in \mathfrak{X}  (S) : \operatorname{div}_{\mu _S}w=0\;\text{ and }\;w\| \partial S\}
\end{align*}
the Lie algebras of $\Diff_{\vol}(M)$ and $\Diff_{\vol}(S)$, respectively.
The infinitesimal actions of $v\in\X_{\vol}(M)$ and $w\in\X_{\vol,\|}(S)$ are given by
\[
v_{\Emb_{\vol}(S,M)}(\varphi )=v\o \varphi \quad\text{and}\quad w_{\Emb_{\vol}(S,M)}(\varphi )=T\varphi \o w.
\]
Note that from \eqref{Tangent_space_vol}, each $ v _\varphi \in T_ \varphi \Emb_{\rm vol}(S,M)$ reads $v _\varphi = v \circ \varphi $, where $v$ is a divergence free vector field on $ \varphi (S)$. 

\begin{remark}{\rm
There is an analogous result to the transitivity of the Lie algebra action of 
$\mathfrak{X}(M) $ on $\Emb(S,M)$, see \cite{H76}, in the case of $\Emb_{\vol}(S,M)$, but this requires $H^{n-1}(S)=0$.
In this case, the Lie algebra $\X_{\vol}(M)$ of divergence free vector fields acts transitively on $\Emb_{\vol}(S,M)$. Indeed, the divergence free vector field $v$ on $\ph(S)$ admits a potential, so it can be easily extended to a divergence free vector field on $M$.}
\end{remark}
\color{black}
%%%%%

\paragraph{Non-linear Grassmannians.} {The  non-linear Grassmannian $\Gr^S(M)$ of submanifolds of $M$ of type $S$ is a Fr\'echet manifold, see \cite{GBVi2014}, where the result is proved for any compact manifold with smooth boundary and with $ \operatorname{dim}S\leq \operatorname{dim}M$ (see \cite{KrMi97} for the case  $\dim S<\dim M$, $ \partial S=\varnothing$).
Recall that here $ \operatorname{dim}S= \operatorname{dim}M$ and $ \partial S\neq \varnothing $. In this case}, the tangent space at $D\in\Gr^S(M)$ can be identified with
\[
T_D\Gr^S(M)=T_{\pa D}\Gr^{\pa S}(M)=\Ga(T(\pa D)^\perp),
\]
where $T(\pa D)^\perp=TM|_{\pa D}/T(\pa D)$ denotes
the normal bundle of the codimension one submanifold $\pa D\subset M$.
If $M$ is endowed with a fixed Riemannian metric $g$,
since $\pa D$ is oriented, we can identify the space of sections of the normal bundle $T(\pa D)^\perp$ with $C^\oo(\pa D)$.

\medskip

Using the submersion 
$$\vol:\Gr^S(M)\to\RR,\quad\vol(D)=\int_D\mu_M,$$
we can express the non-linear Grassmannian $\Gr^S_0 (M)$ of type $S$ submanifolds of $M$ 
with the same total volume as $S$, defined in \eqref{defi}, as
\begin{equation*}
\operatorname{Gr}_0^S(M)
=\vol^{-1}\left(\int_S\mu_S\right).
\end{equation*}
By a regular value theorem valid in the Fr\'echet context (see Theorem III.11 in \cite{NW}), extracted from Gloeckner's implicit function theorem (Theorem 2.3 in \cite{Gloeckner}) it can be shown that $\Gr^S_0 (M)$ is a codimension one submanifold of $\Gr^S(M)$.
.

Using again the Riemannian metric $g$ on $M$,
the tangent space is seen to be 
\[
T_D\Gr_0^S(M)=\left\{f\in C^\oo(\pa D):\int_{\pa D} f\ga_{\pa D}=0\right\},
\]
where $\ga_{\pa D}$ denotes the volume form induced by $g$.

\paragraph{Principal bundles.} As shown in \cite{GBVi2014}, the non-linear Grassmannian $\Gr^S(M)$  
is the base space of a principal bundle with structure group $\Diff(S)$:
\begin{equation}\label{prima}
\pi:\ph\in\Emb(S,M)\mapsto D=\ph(S)\in\Gr^S(M).
\end{equation}
The projection $\pi$ restricted to the set of volume preserving embeddings
$$ 
\operatorname{Emb}_{\vol}(S,M):=\{ \ph\in \operatorname{Emb}(S,M)\mid \ph ^\ast \mu_M= \mu_S \}
$$
takes values in $\Gr_0^S(M)$ because 
$\int_{\ph(S)}\mu_M=\int_S\ph^*\mu_M=\int_S\mu_S$ 
for all $\ph\in\Emb_{\vol}(S,M)$.
As shown in \cite{GBVi2014}, we get a principal $\Diff_{\vol}(S)$ bundle over the non-linear Grassmannian 
$\operatorname{Gr}_0^S(M)$ of all type $S$ submanifolds of $M$ of same volume as $S$:
\begin{equation}\label{princ}
\pi:\ph\in\Emb_{\vol}(S,M)\mapsto D=\ph(S)\in\Gr_0^S(M).
\end{equation}

\color{black}

\paragraph{Cotangent bundles.}
The regular cotangent space to $\Emb(S,M)$ at $\varphi $ is the space of one-forms on $M$ along $\varphi $, i.e., 
\[
T^*_\varphi \Emb(S,M)=\Gamma ( \varphi ^\ast T^*M)= \{\al_\varphi \in C^\infty(S,T^\ast M) : \alpha _\varphi (s) \in T^*_{ \varphi (s)}M,\;\forall\,s \in S\}.
\]
The duality pairing between $T _ \varphi \Emb(S,M)$ and $T^*_ \varphi \Emb(S,M)$ is
\[
\left\langle \al_\varphi ,v_\varphi \right\rangle =\int_S(\al_\varphi(s)\!\cdot\!v_\varphi (s) )\mu_S.
\]
As in the case of the tangent bundle, each $ \alpha _ \varphi $ can be written as $ \alpha _ \varphi = \alpha \circ \varphi $ for some $ \alpha \in \Omega ^1 (M)$.

\medskip

In order to describe the cotangent space to $\Emb_{\vol}(S,M)$ at $\varphi $, we shall use the inclusion 
$\iota:\Emb_{\vol}(S,M)\to\Emb(S,M)$.

\paragraph{Notations.} Given an embedding $ \varphi \in \Emb(S,M)$ and a $k$-form $ \omega  _ \varphi $ along $ \varphi $, i.e., $ \omega  _\varphi \in \Gamma ( \varphi ^\ast \Lambda ^k M)$, we define the pull-back $ \varphi ^\ast \omega  _ \varphi \in \Omega ^k (S)$, by 
\begin{equation}\label{Pull_back_phi} 
(\varphi ^*\omega  _\varphi )(s)(u^1_s,\dots,u^k_s):=\omega  _\varphi (s)\left( T_s\varphi (u^1_s),\dots,T_s\varphi (u^k_s)\right),\;\; \forall \;u^1_s,\dots,u^k_s\in T_sS.
\end{equation} 
An example of a form along $\varphi $ is given by the contraction $ \mathbf{i} _ { v _ \varphi } \mu _M$, where $ v_ \varphi \in T_ \varphi \operatorname{Emb}(S,M)$ and $ \mu _M$ is a volume form on $M$. Explicitly, this contraction reads
\[
\mathbf{i} _ { v _ \varphi } \mu _M(s)(v^1_{ \varphi (s)}, \ldots, v^{n-1}_{ \varphi (s)})= \mu _M( \varphi (s))\left(  v _ \varphi (s), v^1_{ \varphi (s)}, \ldots, v^{n-1}_{ \varphi (s)}\right),
\]
for all $v^1_{ \varphi (s)}, \ldots, v^{n-1}_{ \varphi (s)}\in T_{ \varphi (s)}M$.
From \eqref{Pull_back_phi}, its pull-back is the $(n-1)$-form on $S$ given by
\[
\varphi ^\ast (\mathbf{i} _ { v _ \varphi } \mu _M)(s)(u^1_{ s}, \ldots, u^{n-1}_{ s})=  \mu _M( \varphi (s))\left(  v _ \varphi (s), T_s\varphi (u^1_s),\dots,T_s\varphi (u^{n-1}_s)\right),
\]
for all $u^1_s,\dots,u^{n-1}_s\in T_sS$.

\medskip
 
\begin{lemma}\label{lem}
The regular cotangent space at $\varphi \in\Emb_{\vol}(S,M)$ can be identified with the quotient space
\begin{equation}\label{description_Tstar} 
T^*_\varphi \Emb_{\vol}(S,M)=\left.\Ga(\varphi ^*T^*M)\right/\left \{\al_\varphi :\varphi ^*\al_\varphi =\mathbf{d} h\text{ for }h\in C^\oo(S), h|_{\pa S}=0\right \}.
\end{equation} 
\end{lemma}

\begin{proof}
It is enough to show that 
\[
\ker (T^*_\varphi \iota)=\{\al_\varphi :\varphi ^*\al_\varphi =\mathbf{d} h\text{ for }h\in C^\oo(S), h|_{\pa S}=0\},
\]
where $T^*_\varphi \iota:T^*_\varphi \Emb(S,M)\to T^*_\varphi \Emb_{\vol}(S,M)$.
Indeed, for all $v_\varphi \in T_\varphi \Emb_{\vol}(S,M)$,
and for all $\al_\varphi \in T^*_\varphi \Emb(S,M)=\Ga(\varphi ^*T^*M)$, we have
\[
\left\langle T^*_\varphi \iota\cdot \al_\varphi , v_\varphi \right\rangle =\left\langle \al_\varphi ,T_\varphi \iota\cdot v_\varphi \right\rangle=\int_S(\al_ \varphi  \!\cdot \!v_\varphi )\mu_S=\int_S(\al_\varphi \!\cdot\! v_\varphi) \varphi ^*\mu_M
=\int_S\varphi ^*\al_\varphi \wedge \varphi ^*\mathbf{i} _{v_\varphi }\mu_M.
\]
So, $\al_\varphi \in\ker(T_\varphi ^*\iota)$ if and only if $\int_S\varphi ^*\al_\varphi \wedge \varphi ^*\mathbf{i} _{v_\varphi }\mu_M=0$ for all $v_\varphi $ such that $\mathbf{d} \left( \varphi ^*\mathbf{i} _{v_\varphi }\mu_M\right) =0$, i.e.,  for all closed $(n-1)$-forms
$\varphi ^*\mathbf{i} _{v_\varphi }\mu_M$ on $S$. From the Hodge decomposition $ \Omega ^k (S)= \mathbf{d} \Omega ^{k-1}_n(S) \oplus \Omega ^k_{ \delta -cl}(S)$ for one-forms, we get that $\al_\varphi \in\ker \left( T_\varphi ^*\iota\right) $ if and only if $\varphi ^*\al_\varphi $ is the differential of 
a smooth function $h$ on $S$ that vanishes on $\pa S$.
\end{proof}

\medskip

{We will denote by $[ \alpha _\varphi ]$ the element in the quotient space \eqref{description_Tstar}}

%%%%%%%%%%%%%%%

\section{Dual pairs for free boundary fluids}\label{5}

{In this section, we prove the dual pair properties for the momentum map $ \mathbf{J} _R$ and the quotient map $ \pi _R$ associated to free boundary perfect fluids, given in \eqref{DP_free_boundary}. It is worth noting that to obtain this result, we first prove the dual pair property for the pair of momentum maps $ \mathbf{J} _L$ and $ \mathbf{J} _R$, by using the results of Section \ref{deux}. As we shall show, this pair of momentum maps is not associated to mutually completely orthogonal actions. Then, using again a result from Section \ref{deux}, we deduce the dual pair property for \eqref{DP_free_boundary}.}

\subsection{A pair of momentum maps}

Consider the left and right actions of $\Diff_{\vol}(M)$ and $\Diff_{\vol}(S)$ on $\Emb_{\vol}(S,M)$ and the associated cotangent lifted action on $T^*\Emb_{\vol}(S,M)$. Since the actions commute, it follows from Corollary \ref{finocchio} that the associated momentum maps
\[
\X_{\vol}(M)^*\stackrel{\J_L}{\longleftarrow}T^*\Emb_{\vol}(S,M)\stackrel{\J_R}{\longrightarrow}\X_{\vol,\|}(S)^*
\]
form a weak dual pair.
We now compute these momentum maps by using formula \eqref{cot_momap}.

For the $\Diff_{\vol}(M)$-action, given  $[\al_\varphi ]\in T^*\Emb_{\vol}(S,M)$, we have
\begin{equation}\label{jl}
\langle\J_L([\al_\varphi ]),v\rangle=\left\langle [\al_\varphi ],v\o \varphi \right\rangle =\int_S\al_\varphi \! \cdot \!(v\o \varphi )\mu_S
=\int_{\varphi (S)}(\al_\varphi \circ \varphi ^{-1} \! \cdot \!v) \mu_M,
\end{equation} 
for all $v\in\X_{\vol}(M)$. Therefore, 
$\J_L:T^*\Emb_{\vol}(S,M)\to\X_{\vol}(M)^*$ takes values in the singular dual 
of the Lie algebra of divergence free vector fields on $M$.

Let us identify the regular dual to $ \mathfrak{X}  _{\vol,\|}(S)$ with the quotient space $\Om^1(S)/\mathbf{d} \Omega ^0 (S)$. 
For the $\Diff_{\vol}(S)$-action we compute 
\[
\langle\J_R([\al_\varphi ]),w\rangle=\left\langle [\al_\varphi ],T\varphi \o w\right\rangle =\int_S\al_\varphi \! \cdot \!(T\varphi \o w)\mu_S
=\int_S(\varphi ^*\al_\varphi \! \cdot \!w)\mu_S=\langle[\varphi ^*\al_\varphi ],w\rangle,
\]
for all $w\in\X_{\vol,\|}(S)$. Therefore, the right cotangent momentum map $\J_R$ takes values in the regular dual of $\X_{\vol,\|}(S)$ and reads
\begin{equation}\label{jr}
\J_R:T^*\Emb_{\vol}(S,M)\to\X_{\vol,\|}(S)^*,\quad
\J_R([\al_\varphi ])=[\varphi ^*\al_\varphi ].
\end{equation}

%%%

\subsection{Transitivity}

In the case of the dual pair on $T^*\Emb(S,M)$ associated to the EPDiff equation, see \cite{GBVi2012}, 
we had to restrict the cotangent  momentum maps
to an open subset of the cotangent bundle in order to show that the weak dual pair is a dual pair.
We denoted by $T_\varphi ^*\Emb(S,M)^\times $ the open subset  of $T_\varphi ^*\Emb(S,M)=\Ga(\varphi ^*T^*M)$
consisting of those $1$-forms along $\varphi $ which are everywhere non-zero on $S$.
Consequently we consider here the open subset
\[
T^*_\varphi \Emb_{\vol}(S,M)^\x=T_\varphi ^*\Emb(S,M)^\times/\{\al_\varphi :\varphi ^*\al_\varphi =\mathbf{d} h\;\text{ for }\;h\in C^\oo(S),\; h|_{\pa S}=0\}.
\]
The restrictions of the momentum maps to this open subset will be denoted again by $\J_L$ and $\J_R$.

\begin{proposition}\label{fenicul}
The group $\Diff_{\vol}(S)$ of volume preserving diffeomorphisms of $S$ acts transitively on the level sets
of the cotangent momentum map $\J_L:T^*\Emb_{\vol}(S,M)^\x\to\X_{\vol}(M)^*$ given by \eqref{jl}.
\end{proposition}

\begin{proof}
Let $[\al_\varphi ],[\al'_{\varphi '}]\in T^*\Emb_{\vol}(S,M)^\x$ (see Lemma \ref{lem}) and assume that they belong to the same level set of the momentum map $\J_L$.
By \eqref{jl} this means that for all $v\in\X_{\vol}(M)$,
\begin{equation}\label{level}
\int_S\al_\varphi \! \cdot \!(v\o \varphi )\mu_S=\int_S\al'_{\varphi '}\! \cdot \!(v\o \varphi ')\mu_S.
\end{equation}
Let $\ga:=\al_\varphi\circ  \varphi ^{-1}\in\Ga(T^*M|_{\varphi (S)})$ and $\ga':=\al'_{\varphi '}\o\varphi '^{-1}\in\Ga(T^*M|_{\varphi '(S)})$.
Using the fact that both $\varphi $ and $\varphi '$ are volume preserving, we deduce from \eqref{level} that
\[
\int_{\varphi (S)}\ga\wedge \mathbf{i} _v\mu_M =\int_{\varphi '(S)}\ga'\wedge \mathbf{i} _v\mu_M 
\]
for all divergence free vector fields $v$ on $M$, hence
$
\int_{\varphi (S)}\ga\wedge\nu=\int_{\varphi '(S)}\ga'\wedge\nu
$
for all closed $(n-1)$-forms $\nu$ on $M$. 

The embeddings $\varphi $ and $\varphi '$ have the same image in $M$. Otherwise, by choosing $\nu$ with appropriately small support
we can achieve that one member of the above equality is zero, while the other member is not zero.
Here we need the fact that $\ga$ \resp $\ga'$ is everywhere non-zero on $\varphi (S)$ \resp $\varphi '(S)$,
which follows from $\al_\varphi ,\al'_{\varphi '}\in T^*\Emb(S,M)^\x$. 
Since $\varphi (S)=\varphi '(S)$, we can define $ \psi := \varphi ^{-1} \circ \varphi '\in \operatorname{Diff}(S)$. We have $ \psi ^\ast \mu _S = (\varphi ') ^\ast ( \varphi ^{-1} ) ^\ast \mu _S =(\varphi ') ^\ast  \mu _M = \mu _S$, because $ \varphi , \varphi '$ are volume preserving, which proves that $\ps$ is volume preserving, i.e.,  $\ps\in\Diff_{\vol}(S)$.

We define the one-form along $\varphi $ given by $\be_\varphi :=\al'_{\varphi '}\o\ps^{-1}$. In particular
 $[\al'_{\varphi '}]=[\be_\varphi \o\ps]=\ps\cdot[\be_\varphi ]$, where the dot denotes the cotangent action of $\Diff_{\vol}(S)$
on $T^*\Emb_{\vol}(S,M)^\x$. Using \eqref{level}, as well as some of the identities above, we compute
\begin{align*}
\int_S\varphi ^*\al_\varphi \wedge \varphi ^*(\mathbf{i} _v\mu_M)&=\int_S\al_\varphi \! \cdot \!(v\o \varphi )\mu_S
=\int_S\al'_{\varphi' }\! \cdot \!(v\o \varphi')\mu_S
=\int_S(\be_\varphi \o\ps)\! \cdot \!(v\o \varphi \o\ps)\mu_S\\
&=\int_S\be_\varphi \! \cdot \!(v\o \varphi )\mu_S
=\int_S\varphi ^*\be_\varphi \wedge \varphi ^*(\mathbf{i} _v\mu_M),
\end{align*}
for all $v\in\X_{\vol}(M)$, hence for all closed $(n-1)$-forms $\varphi ^*(\mathbf{i} _u\mu_M )$ on $S$.
From the Hodge decomposition, we conclude that the one-forms $\varphi ^*\al_\varphi $ and $\varphi ^*\be_\varphi $ on $S$ differ by an exact form $\mathbf{d} h$ with $h|_{\pa S}=0$,
\ie $[\al_\varphi ]=[\be_\varphi ]$. Now, we have $[\al'_{\varphi '}]=\ps\cdot[\be_\varphi ]=\ps\cdot[\al_\varphi ]$ 
and  $\ps\in\Diff_{\vol}(S)$, hence we get the desired result. 
\end{proof}

\medskip

The following result is a consequence of Corollary \ref{finocchio} and Proposition \ref{fenicul}:

\begin{corollary}\label{dupa}
The pair of momentum maps
\[
\X_{\vol}(M)^*\stackrel{\J_L}{\longleftarrow}T^*\Emb_{\vol}(S,M)^\x\stackrel{\J_R}{\longrightarrow}\X_{\vol,\|}(S)^*
\]
is a dual pair.
\end{corollary}

We will show below that the commuting actions of $\Diff_{\vol}(S)$ and $\Diff_{\vol}(M)$ are not mutually completely orthogonal (in the sense of \cite{LM}, see Remark \ref{mut_compl_orth}),
namely only the $\Diff_{\vol}(S)$-orbits are the symplectic orthogonals of the $\Diff_{\vol}(M)$-orbits, not vice-versa.
Thus we have an example of a dual pair of momentum maps
$\h^*\stackrel{\J_L}{\longleftarrow}(M,\om)\stackrel{\J_R}{\longrightarrow}\g^*$
that satisfies
$\g_M=(\h_M)^\om$, but not $\h_M=(\g_M)^\om$ (see Remark \ref{oneisenough}).

\begin{proposition}
The action of the Lie algebra of divergence free vector fields $\X_{\vol}(M)$
on level sets of the cotangent momentum map 
$\J_R:T^*\Emb_{\vol}(S,M)\to\X_{\vol,\|}(S)^*$ given by \eqref{jr}
is not transitive $($hence the action of the volume preserving diffeomorphism group $\Diff_{\vol}(M)$ is not transitive$)$.
\end{proposition}

\begin{proof}
We will show that  the tangent space at $[\al_\varphi ]$
to the $\Diff_{\vol}(M)$-orbit through $[\al_\varphi ]$, namely $\X_{\vol}(M)_{T^*\Emb_{\vol}(S,M)}([\al_\varphi ])$, is strictly included in the kernel of the tangent map $T_{[\al_\varphi ]}\J_R$.

Given $[\al_\varphi ],[\be_\varphi ]\in T^*_\varphi \Emb_{\vol}(S,M)$, we consider the vertical lift 
$$
\operatorname{Ver} _{[\al_\varphi ]}([\be_\varphi ]):= \frac{d}{dt}\Big|_{t=0}[\al_\varphi +t\be_\varphi ]\in T_{[\al_\varphi ]}(T^*\Emb_{\vol}(S,M)).
$$ 
If $\J_R([\be_\varphi ])=[\varphi ^*\be_\varphi ]=0$, then the vertical lift $\operatorname{Ver} _{[\al_\varphi ]}([\be_\varphi ])$
belongs to the kernel of $T_{[\al_\varphi ]}\J_R$.
This is the case for $ [\beta _\varphi ]$ such that
\begin{equation}\label{beta_h}
\varphi ^*\be_\varphi =\mathbf{d} h,
\end{equation}
for an arbitrary $ h \in C^\infty(S)$.

We assume by contradiction that the vertical lift 
$\operatorname{Ver}_{[\al_\varphi ]}([\be_\varphi ])$ is tangent to the $\Diff_{\vol}(M)$-orbit through $[\al_\varphi ]$,
\ie there exists $v\in\X_{\vol}(M)$ such that
\begin{equation}\label{unu}
\operatorname{Ver} _{[\al_\varphi ]}([\be_\varphi ])=v_{T^*\Emb_{\vol}(S,M)}([\al_\varphi ]).
\end{equation}
The tangent map of the canonical projection $T^*\Emb_{\vol}(S,M)\to\Emb_{\vol}(S,M)$ applied to \eqref{unu}
leads to $0=v_{\Emb_{\vol}(S,M)}(\varphi )$,
so the divergence free vector field $v$ on $M$
must vanish on the image $\varphi (S)$ of the embedding $\varphi $.

There is an induced action of the Lie subalgebra of divergence free 
vector fields on $M$ that vanish on $\varphi (S)$ on the vector space $T^*_\varphi \Emb_{\vol}(S,M)$ by $v\cdot\al_\varphi =\al_\varphi \o\nabla v$.
It can be used to express the infinitesimal generator as
\begin{equation}\label{doi}
v_{T^*\Emb_{\vol}(S,M)}([\al_\varphi ])=\operatorname{Ver} _{[\al_\varphi ]}([\al_\varphi \o\nabla v]).
\end{equation}
The identities \eqref{unu} and \eqref{doi} lead to 
$[\al_\varphi \o\nabla v]=[\be_\varphi ]$, so {there exists $k \in C^\infty(S)$, $k|_{ \partial S}=0$ such that $\varphi ^*(\al_\varphi \o\nabla v)+ \mathbf{d} k=\varphi ^*\be_\varphi \in\Om^1(S)$}.
This is a contradiction because the left hand side vanishes on $T(\pa S)$,
while {the function $h$ in \eqref{beta_h} can be chosen such that} the right hand side doesn't.
\end{proof}

\subsection{Dual pair for perfect free boundary fluid}

In this subsection, we will show that the pair of Poisson mappings \eqref{DP_free_boundary} associated to free boundary perfect fluids, is a dual pair. This pair of momentum map is of the form \eqref{reduction} and we shall use Proposition \ref{modulog} applied to the dual pair obtained in Corollary \ref{dupa} to show that \eqref{DP_free_boundary} is a dual pair. 
We first need the following result.

\begin{theorem}\label{theo1}
The quotient space $T^*\Emb_{\vol}(S,M)^\x/\Diff_{\vol}(S)$ can be endowed with a smooth manifold structure such that the projection 
\[
\pi_R:T^*\Emb_{\vol}(S,M)^\x\to T^*\Emb_{\vol}(S,M)^\x/\Diff_{\vol}(S)
\]
is a smooth map.
\end{theorem} 
\begin{proof} {We will use the smooth principal bundle structure \eqref{princ}
on $\Emb_{\vol}(S,M)$.}

Consider the Riemannian metric on $\Emb_{\vol}(S,M)$ given by
\[
\langle
v_\varphi ,w_\varphi \rangle_\varphi :=\int_Dg(\varphi (x))(v_\varphi (x),w_\varphi (x))\mu_M(x),
\]
where $g$ is a Riemannian metric on $M$ and $ \mu _M$ is the associated volume form. The principal connection associated to this Riemannian metric is called the Neumann connection (\cite{LeMaMoRa1986}) and is given by
\[
\mathcal{A} (\varphi )(v_\varphi )= \varphi ^*\mathbb{P}_{\varphi (S)}(u_\varphi \circ \varphi ^{-1})\in \mathfrak{X}_{\vol,\|}(S),
\]
where $\mathbb{P}_{\varphi (S)}$ is the projector onto the first component of the Helmholtz decomposition
\[
\mathfrak{X}(\varphi (S))=\mathfrak{X}_{\vol,\|}(\varphi (S))\oplus\operatorname{grad}\,\mathcal{F}(\varphi (S)).
\]
Here $\mathfrak{X}(\varphi (S))$ denotes the space of all smooth vector fields on the manifold with boundary $ \varphi (S)$ (not necessarily parallel to the boundary).
Recall that since $v _\varphi  \in T_\varphi  \operatorname{Emb}_{\vol}(S,M)$, the vector field $v_\varphi \circ \varphi ^{-1}$ is divergence free but not necessarily parallel to the boundary.

Using the connection $ \mathcal{A} $, we have an isomorphism
\begin{equation}\label{izo}
\big(  T\Emb_{\vol}(S,M)\big) /\Diff_{\vol}(S) \rightarrow T \operatorname{Gr}_0^{S}(M) \oplus \widetilde{\mathfrak{X}_{\vol ,\|}(S)}
\end{equation}
\[
[v _\varphi ] \mapsto \left( T_\varphi  \pi (v _\varphi  ), [ \varphi , \mathcal{A} (\varphi )(v _\varphi  )]_{\Diff_{\vol}(S)}\right)
\]
covering the identity on $\operatorname{Gr}_0^{S}(M)$.
Here
\[
\widetilde{\mathfrak{X}_{\vol ,\|}(S)}:= \left( \Emb_{\vol}(S,M) \times  \mathfrak{X}_{\vol ,\|}(S)\right) /\Diff_{\vol}(S)
\]
is the adjoint bundle, where the quotient is taken relative to the right action of $ \eta \in \Diff_{\vol}(S)$ on $\Emb_{\vol}(S,M) \times  \mathfrak{X}_{\vol ,\|}(S)\ni ( \varphi ,v)$ given by
\[
(\varphi , v) \mapsto ( \varphi  \circ \eta , \eta ^\ast v).
\]
{The fact that $\pi:\Emb_{\vol}(S,M)\to \Gr_0^S(M)$ is a principal bundle 
implies that $\widetilde{\mathfrak{X}_{\vol ,\|}(S)}$ is a smooth vector bundle
over $\Gr_0^S(M)$.} Thus, by using the isomorphism \eqref{izo}, we get
that the quotient space $\big(  T\Emb_{\vol}(S,M)\big) /\Diff_{\vol}(S) $ is  a smooth vector bundle and hence a manifold. 

The regular dual of the vector bundle \eqref{izo} can be identified with
the quotient space $\big(  T^\ast \Emb_{\vol}(S,M)\big) /\Diff_{\vol}(S) $,
where $T^\ast \Emb_{\vol}(S,M)$ is the regular dual of
the tangent bundle,
% $T\Emb_{\vol}(S,M)$,
so it is a manifold. 
The $\Diff_{\vol}(S)$-action preserves the open subset $T^\ast \Emb_{\vol}(S,M)^\times$ of $T^\ast \Emb_{\vol}(S,M)$.
Thus  $\big(  T^\ast \Emb_{\vol}(S,M)^\times\big) /\Diff_{\vol}(S) $ is 
an open subset of the manifold $\big(  T^\ast \Emb_{\vol}(S,M)^\times\big) /\Diff_{\vol}(S) $,
hence also a manifold.
\end{proof}

\medskip

Now Proposition \ref{modulog} ensures that the Lagrange-to-Euler map $ \pi _R $ 
and the momentum map $ \mathbf{J} _R $ for the free boundary fluid form a dual pair
of Poisson maps.

\begin{theorem}
The pair of Poisson maps
\[
T^*\Emb_{\vol}(S,M)^\x/\Diff_{\vol}(S)\;\stackrel{\pi_R }{\longleftarrow}\;T^*\Emb_{\vol}(S,M)^\x\;\stackrel{\J_R}{\longrightarrow}\;\X_{\vol,\|}(S)^*
\]
is a dual pair.
\end{theorem}

%%%%%%%%%

{\footnotesize

\bibliographystyle{new}

\addcontentsline{toc}{section}{References}
\begin{thebibliography}{GBTV}

\bibitem[Arnold(1966)]{Arnold1966}
Arnold, V.~I. [1966], Sur la g\'{e}om\'{e}trie diff\'{e}rentielle des
groupes de {L}ie de dimenson infinie et ses applications \`{a}
l'hydrodynamique des fluides parfaits, \textit{Ann. Inst. Fourier, Grenoble}
{\bf 16} (1966), 319--361.

\bibitem[Cushman and Rod(1982)]{CuRo1982}
Cushman, R. and D. Rod [1982], Reduction of the semisimple 1:1 resonance, \textit{Physica D} \textbf{6}, 105--112.

\bibitem[Ebin and Marsden(1970)]{EbMa1970}
Ebin, D.~G. and J.~E. Marsden [1970], Groups of diffeomorphisms and the motion
of an incompressible fluid, {\em Ann. of Math.} \textbf{92}, 102--163.

\bibitem[Gay-Balmaz, Marsden, and Ratiu(2013)]{GBMR2012}
Gay-Balmaz, F., J.~E. Marsden, and T.~S. Ratiu [2012], Reduced variational formulations in free boundary continuum mechanics, \textit{J. Nonlin. Sci.}, \textbf{22}(4), 463--497.

\bibitem[Gay-Balmaz and Vizman(2012)]{GBVi2012}
Gay-Balmaz, F. and C. Vizman [2012], Dual pairs in fluid dynamics,
\textit{Ann. of Global Analysis and Geometry}, 
%DOI: 10.1007/s10455-011-9267-z,
{\bf 41}, 1--24.

\bibitem[Gay-Balmaz and Vizman(2014)]{GBVi2014}
Gay-Balmaz, F. and C. Vizman [2014], 
Principal bundles of embeddings and nonlinear Grassmannians, preprint.

\bibitem[Gl\"ockner(2003)]{Gloeckner}
Gloeckner, H. [2003],
Implicit functions from topological vector spaces to Banach spaces,
{\it Israel Journal Math.}, {\bf 155} 205--252.

\bibitem[Golubitsky and Stewart(1987)]{GoSt1987}
Golubitsky, M. and I. Stewart [1987], Generic bifurcation of Hamiltonian systems with symmetry, \textit{Physica D} \textbf{24}, 391--405.

\bibitem[Haller and Vizman(2004)]{HaVi2004}
Haller, S. and C. Vizman [2004],
Non--linear Grassmannians as coadjoint orbits,
\textit{Math. Ann.} \textbf{329}, 771--785.

\bibitem[Hirsch(1976)]{H76}
Hirsch, M.~W. [1976], {\it Differential topology}, Graduate Texts in Math. \textbf{33}, Springer.

\bibitem[Holm and Marsden(2004)]{HoMa2004}
Holm, D.~D. and J.~E. Marsden [2004], Momentum maps and measure-valued solutions (peakons, filaments and sheets) for the EPDiff equation, in \textit{The Breadth of Symplectic and Poisson Geometry}, A Festshrift for Alan Weinstein, 203-235, Progr. Math., \textbf{232}, J.~E. Marsden and T.~S. Ratiu, Editors, Birkh\"auser Boston, Boston, MA, 2004.

\bibitem[Holm and Tronci(2009)]{HoTr2009}
Holm, D.~D. and C. Tronci [2009], The geodesic Vlasov equation and its integrable moment closures
solutions, \textit{The Journal of Geometric Mechanics} \textbf{1}(2), 181--208.

\bibitem[Iwai(1985)]{Iw1985}
Iwai, T. [1985], On reduction of two degree of freedom Hamiltonian systems by an $S^1$ action, and $SO_0(1,2)$ as a dynamical group, \textit{J. Math. Phys.} \textbf{26}, 885--893.

\bibitem[Kriegl and Michor(1997)]{KrMi97}
Kriegl, A., and P.~W. Michor [1997],
{\it The Convenient Setting of Global Analysis},
{Mathematical Surveys and Monographs} {\bf 53},
American Mathematical Society, Providence, RI.

\bibitem[Lewis et al.(1986)]{LeMaMoRa1986}
Lewis, D., J.~E. Marsden, R. Montgomery, and T.~S. Ratiu [1986], The Hamiltonian Structure for Dynamic Free Boundary Problems, \textit{Physica D}  \textbf{18}, 391--404.

\bibitem[Libermann and Marle(1987)]{LM}
Libermann, P. and C.-M. Marle [1987], 
{\it Symplectic Geometry and Analytical Mechanics}, D. Reidel Publishing Company.

\bibitem[Marsden(1987)]{Ma1987}
Marsden, J. E. [1987], Generic bifurcation of Hamiltonian systems with symmetry, 
\textit{appendix to Golubitsky and Stewart, Physica D} \textbf{24}, 391--405.

\bibitem[Marsden and Ratiu(1999)]{MaRa99}
Marsden, J.~E. and T.~S. Ratiu [1999], \textit{Introduction to mechanics and symmetry},
Second Edition, Springer.

\bibitem[Marsden and Weinstein(1983)]{MaWe83}
Marsden, J. E. and A. Weinstein [1983], Coadjoint orbits, vortices, and Clebsch variables for incompressible fluids, \textit{Phys. D} \textbf{7}, 305--323.

%\bibitem[Morrey(1966)]{Mo1966}
%Morrey, C.~B. [1966], \textit{Multiple Integrals in the Calculus of Variations}, Springer 1966.

\bibitem[Neeb and Wagemann(2008)]{NW}
Neeb, K.-H. and F. Wagemann [2008],
Lie group structures on groups of smooth and holomorphic maps on non-compact manifolds,
{\it Geometriae Dedicata}, {\bf 134} 17--60.

\bibitem[Ortega and Ratiu(2004)]{OrRa04}
Ortega, J.-P. and T.~S. Ratiu [2004], \textit{Momentum maps and
Hamiltonian reduction}, Progress in Mathematics (Boston, Mass.)
\textbf{222} Boston,  Birkh\"auser.

\bibitem[Weinstein(1983)]{We83}
Weinstein, A. [1983], The local structure of Poisson manifolds, \textit{J. Diff. Geom.} \textbf{18}, 523--557.

\end{thebibliography}

\bigskip

FRAN{C}OIS GAY-BALMAZ, 
%{Control and Dynamical Systems, California Institute of Technology 107-81, Pasadena, CA 91125, USA. \texttt{fgbalmaz@cds.caltech.edu} and
{LMD, Ecole Normale Sup\'erieure/CNRS, Paris, France.\\
\texttt{gaybalma@lmd.ens.fr}
\\

CORNELIA VIZMAN, {Department of Mathematics,
West University of Timi\c soara, 
%RO--300223 Timi\c soara, 
Romania.\\
\texttt{vizman@math.uvt.ro}
}
\end{document}